\documentclass[a4paper, 12pt]{amsart}
\usepackage{amsmath,latexsym,amssymb, hyperref}

\newif\ifpdf
\ifx\pdfoutput\undefined
\pdffalse 
\else
\pdfoutput=1 
\pdftrue
\fi

\ifpdf
\usepackage[pdftex]{graphicx}
\else
\usepackage{graphicx}
\fi

\newtheorem{lemma}{Lemma}[]
\newtheorem{theo}{Theorem}[]

\newtheorem{corollary}[lemma]{Corollary}

\newtheorem{proposition}[lemma]{Proposition}

\def\X{\mathcal{X}}
\def\vphi{\varphi}
\def\cl{{\rm cl}}
\def\hs{\hat{s}}
\def\ha{\hat{\alpha}}
\def\N{\mathbb{N}}

\usepackage{setspace}

\begin{document}

\title[]{Finite homogeneous geometries}

\author{David M. Evans}

\address{%
Department of Mathematics\\
Imperial College London\\
London SW7~2AZ\\
UK.}

\email{david.evans@imperial.ac.uk}

\date{18 January 2017}

\begin{abstract}  This paper reproduces the text of a part of the Author's DPhil thesis. It gives a proof of the classification of non-trivial, finite homogeneous geometries of sufficiently high dimension which does not depend on the classification of the finite simple groups.
\end{abstract}

\maketitle

\section*{Preamble} This note is essentially a reproduction of Appendix I of the Author's doctoral thesis \cite{DEthesis}, though there are a few new, minor corrections.  In the thesis, the Appendix was a sequel to Chapters 3 - 6 of \cite{DEthesis} which were published as \cite{DEHG}. Thus, the text below is a complement to \cite{DEHG} and we use the notation and terminology of that paper without further comment. The main changes we have made to Appendix I of \cite{DEthesis} are to adapt the references, so that we refer to results in \cite{DEHG} rather than the corresponding results in \cite{DEthesis}. 

The main result of \cite{DEHG} was a proof of the classification of infinite, locally finite, homogeneous geometries which did not depend on the classification of finite simple groups. A different proof, also not relying on the classification of finite simple groups, had previously been given by Zilber (see the references in \cite{DEHG}). In particular, the geometries considered  were of infinite dimension. By contrast, Theorem \ref{thm1} below assumes only that the geometry is of sufficiently large finite dimension (and so can be finite). As the complements of Jordan sets for a finite primitive Jordan group form a homogeneous geometry (with the Jordan group acting as a group of automorphisms), Theorem \ref{thm1} also gives a classification (not relying on the classification of the finite simple groups) of the finite primitive Jordan groups which are not $3$-transitive and which have `enough' (at least 24) different sizes of Jordan sets. 

The final version of the thesis \cite{DEthesis} contains a number of corrections written in by hand and is not available in electronic form. This explains why the Author has re-typed the material rather than making a scanned copy available. The Author thanks Michael Zieve for suggesting that it would be useful to have these results available in a more accessible form. 

It should be noted that Zilber's paper \cite{Z} contains a different proof (also not using the classification of the finite simple groups) of   Theorem \ref{thm1} under the weaker hypothesis that the dimension of $\X$ is at least 7. Of course, using the classification of finite simple groups, we know that dimension at least 2 will suffice here.

\section*{The Finite Case}

We give a proof of:

\begin{theo}\label{thm1} Let $\X$ be a locally finite homogeneous geometry of dimension at least 23 with at least 3 points on a line. Then $\X$ is a (possibly truncated) projective or affine geometry over a finite field.
\end{theo}

The proof of this result is in a series of lemmas, which are tightenings of results which have appeared in \cite{DEHG}. We use the notation of Section 2.4 of that paper, and assume throughout that $\X = (X, \cl)$ is a locally finite homogeneous geometry of dimension at least 3. In particular, note that the dimension of a closed set is one less than the number of elements in a basis of the set and an $i$-flat is a closed set of dimension $i$. Points, lines and planes are closed sets of dimension 0, 1, 2 respectively. We denote by $s_i$ the number of points in an $i$-flat. The parameters $\alpha$ and $\alpha'$ are defined in 2.4 of \cite{DEHG}.

\begin{lemma} If $r \geq 3$, then $s_r \geq (s_2-s_1)^{r-1}/(s_1-1)^{r-2}$.
\end{lemma}

\begin{proof} By (\cite{DEHG}, Lemma 3.2.3) we have
\[ \frac{s_r-s_{r-2}}{s_{r-1}-s_{r-2}} \geq \frac{s_{r-1}-s_{r-3}}{s_{r-2}-s_{r-3}}\]
and so
\[ s_r \geq s_{r-2}+(s_{r-1}-s_{r-2})\frac{s_{r-1}-s_{r-3}}{s_{r-2}-s_{r-3}}.\]
Thus
\[s_r-s_{r-1} \geq \frac{(s_{r-1}-s_{r-2})^2}{s_{r-2}-s_{r-3}}.\]
Using this relation repeatedly, we obtain
\[ s_r-s_{r-1} \geq (s_2-s_1)^{r-1}/(s_1-1)^{r-2}\]
which proves the lemma.
\end{proof}

\begin{lemma} Suppose that $\alpha' = 0$ and $\alpha \geq 1$ and $s_1 \geq 3$. If $\theta, \vphi, \psi$ are as in (\cite{DEHG}, Theorem 4.3; see also p.319), then 
\[ s_r \leq \frac{\vphi^2(\theta\vphi - 4\alpha s_1\psi)^2-\vphi}{4\alpha s_1}\]
implies that $r \leq 19$.
\end{lemma}

\begin{proof}\footnote{Slight changes from final version of thesis} First note that (in the notation of Section 4 of \cite{DEHG}) $$a_{443} = (s_1-1)\alpha^2/s_1$$ and so $\alpha \geq 1$ implies that $\alpha \geq\sqrt{s_1}$. 

Now, 
\begin{multline*}
\vphi = \alpha^2 +s_1^2(s_1-1)^2 - 2\alpha s_1(s_1+1)
= (\alpha - s_1(s_1-1))^2 - 4\alpha s_1.
\end{multline*}
If this is positive, then $\vphi < (\alpha - s_1(s_1-1))^2.$ If it is negative, then $\vphi > -4\alpha s_1$. So in either case we have $\vphi^2 < (\alpha+s_1(s_1-1))^4$. Moreover, the $-\vphi$ term in the numerator of the above expression makes a contribution of at most $+1$; the upper bound estimates below are sufficiently crude that we may ignore this.

Also, from the calculations preceding Theorem 4.3 of \cite{DEHG}, 
\begin{multline*}
(\theta\vphi - 4\alpha s_1\psi)^2 \\
= (\alpha-s_1(s_1-1))^2(\theta(\alpha-s_1(s_1-1))+4\alpha s_1^2(s_1-1))^2 \\
= (\alpha-s_1(s_1-1))^2((s_1-s_1^2+2\alpha s_1-\alpha)(\alpha-s_1(s_1-1))+4\alpha s_1^2(s_1-1))^2\\
= (\alpha-s_1(s_1-1))^2(\alpha(2s_1-1)+s_1(s_1-1))^2(\alpha+s_1(s_1-1))^2\\
< (\alpha+s_1(s_1-1))^4(\alpha(2s_1-1)+s_1(s_1-1))^2.
\end{multline*}
Thus
\[ \frac{\vphi^2(\theta\vphi - 4\alpha s_1\psi)^2}{4\alpha s_1} < 
\frac{(\alpha+s_1(s_1-1))^8(\alpha(2s_1-1)+s_1(s_1-1))^2}{4\alpha s_1}.\]

Now,
\[s_2-s_1 = (s_1-1)\alpha +(s_1-1)^2 \geq \alpha+s_1(s_1-1)\]
(as $s_1 \geq 3$ and $\alpha \geq 2$). So
\begin{multline*}\frac{\vphi^2(\theta\vphi - 4\alpha s_1\psi)^2}{4\alpha s_1} < 
\frac{(s_2-s_1)^8(\alpha(2s_1-1)+s_1(s_1-1))^2}{4\alpha s_1}\\
= \frac{(s_2-s_1)^8(2\alpha+\alpha/(s_1-1)+s_1)^2(s_1-1)^2}{4\alpha s_1}\\
< \frac{25}{16}\frac{(s_2-s_1)^8(\alpha+s_1)^2(s_1-1)^2}{\alpha s_1}\\
= \frac{25}{16}\frac{(s_2-s_1)^8(s_2-1)^2}{\alpha s_1}<(s_2-s_1)^{10}.
\end{multline*}
As $(s_2-s_1) \geq (s_1-1)^2$, we have
\[ (s_2-s_1)^{10} \leq \frac{(s_2-s_1)^{19}}{(s_1-1)^{18}}.\]
Thus, by Lemma 1, if $r$ is as in the statement of the lemma, then $r \leq 19$.

\end{proof}

\begin{lemma} Suppose that $\alpha' = 1$ and $\alpha > 1$ and $s_1 \geq 3$. If $\beta = \alpha-1$ and
\[s_r \leq \frac{(4\beta s_1+(s_1^2-\beta)^2)^2(s_1^2-\beta)^2(s_1^2+\beta)^2(s_1^2-\beta+2s_1\beta)^2}{4\beta s_1},\]
then $r \leq 16$.
\end{lemma}

\begin{proof} First, note that in the notation of Section 5 of \cite{DEHG}, $$a_{452} = (s_1-1)\beta/s_1$$ and so $s_1$ divides $\beta$. Now,
\[(4\beta s_1+(s_1^2-\beta)^2) = s_1^4+\beta^2- \beta s_1(2s_1 - 4).\]
If this is positive, it is less than $s_1^4 + \beta^2$. If it is negative it is greater than $-2s_1^2\beta$. Thus, in either case,
\[ (4\beta s_1+(s_1^2-\beta)^2)^2 < 2(s_1^4+\beta^2)^2.\]
Also, 
\[ 0 < s_1^2-\beta+2s_1\beta <(s_1+\beta)^2,\]
and so
\[s_r < \frac{(s_1^4+\beta^2)^2(s_1^2-\beta)^2(s_1^2+\beta)^2(s_1+\beta)^2}{2\beta s_1} = \frac{(s_1^8-\beta^4)^2(s_1+\beta)^2}{2\beta s_1}.\]

Since $s_2-s_1 = \beta(s_1-1)+s_1(s_1-1)$ and $\beta \geq s_1$ and $s_1 \geq 3$, it follows that $s_2-s_1 \geq s_1^2+\beta$ and so $s_1^8+\beta^4 < (s_2-s_1)^4$. Thus 
\[ s_r < \frac{(s_2-s_1)^{10}}{2\beta s_1(s_1-1)^2}\]
as $(s_1+\beta) = (s_2-s_1)/(s_1-1)$. Then
\[s_r < \frac{(s_2-s_1)^{16}}{2(s_1-1)^{14}\beta s_1} < \frac{(s_2-s_1)^{16}}{(s_1-1)^{15}}.\]
By Lemma 1, it then follows that $r \leq 16$.
\end{proof}

\begin{corollary} \label{Cor4} Let $\X$ be a locally finite homogeneous geometry of dimension at least 20 with at least 3 points on a line. Then either $\X$ is a (possibly truncated) projective or affine geometry over a finite field, or one of the following conditions holds in $\X$:
\begin{enumerate}
\item[(1)] $\alpha' = 0$, $s_1$ is a square, $\alpha = s_1(\sqrt{s_1}\pm1)^2$ and $s_i$ is a square for $i \geq 3$;
\item[(2)] $\alpha' = 0$, $\alpha = s_1(s_1-1)$ and $(s_1-1)(s_i-1)$ is a square for $i \geq 3$;
\item[(3)] $\alpha' = 1$, $\alpha = s_1^2+1$ and $s_i/s_1$ is a square for $i \geq 3$.
\end{enumerate}
\end{corollary}

\begin{proof} This is deduced from the above lemmas and Theorem 3.2.9, Theorem 4.3, and Theorem 5.2 in \cite{DEHG}, as in Section 6 of \cite{DEHG}.
\end{proof}

\begin{proposition} \label{Prop5} Let $\X= (X, \cl)$ be a locally finite homogeneous geometry and let $Y$ be a closed subset of $X$ such that the localisation $\X_Y$ has dimension at least 3. Let $z \in X \setminus Y$ and $Z = \cl(Y \cup \{z\})$. Consider the following situations:
\begin{enumerate}
\item[(a)] condition (1) holds in $\X_Y$ and (1) holds in $\X_Z$;
\item[(b)] condition (1) holds in $\X_Y$ and (2) holds in $\X_Z$;
\item[(c)] condition (2) holds in $\X_Y$ and (2) holds in $\X_Z$;
\item[(d)] condition (3) holds in $\X_Y$ and (1) holds in $\X_Z$;
\item[(e)] condition (3) holds in $\X_Y$ and (2) holds in $\X_Z$;
\item[(f)] condition (3) holds in $\X_Y$ and (3) holds in $\X_Z$,
\end{enumerate}
where conditions (1), (2) and (3) are as in Corollary \ref{Cor4}. Then (a), (b), (d), (e) and (f) are impossible and (c) is impossible if $s_1 \geq 3$.
\end{proposition}

\begin{proof} We use the notation of the corresponding Proposition in Section 6 of \cite{DEHG}. Conditions (a) and (d) were proved to be impossible in that Proposition.

\textit{(c) is impossible if $s_1\geq 3$:\/} Here $\hs_1= s_1^2$ and $\ha=s_1^2-\hs_1$ and we require that $(s_1-1)(s_3-1)$ be  a square. As $\hs_2 = (s_3-1)/(s_1-1)$ we require that $\hs_2$ be a square. Now,
\[\hs_2 = 1+(\ha+\hs_1)(\hs_1-1) = 1+\hs_1^2(\hs_1-1)=s_1^6-s_1^4+1.\]
Suppose the polynomial $f(x) = x^6-x^4+1$ takes the square integer value $a^2$ at some $t \in \N$. Then 
\[a^2 = t^6-t^4+1 = (t^3-\frac{1}{2}t)^2-\frac{t^2}{4}+1\]
so
\[(2t^3-t+2a)(2t^3-t-2a)= t^2-4.\]
If $t > 2$ then $t^2-4 \neq 0$ and $2t^3-t > t^2-4$. So the above equation has no solution in integers with $t\geq 3$. This proves that (c) is impossible in $s_1 \geq 3$.

\textit{(e) is impossible: \/} Here, $\hs_1 = s_1^2+s_1+1$ and $\ha = \hs_1^2-\hs_1$ and we require that $s_3/s_1$ be a square. Now,
\[\frac{s_3-1}{s_1-1} = \hs_2 = 1+(\ha+\hs_1)(\hs_1-1) = 1+(s_1^2+s_1+1)^2(s_1^2+s_1)\]
so
\[s_3/s_1 = 1+ (s_1^2+s_1+1)^2(s_1-1)^2 = s_1^6+2s_1^5+2s_1^4-2s_1^2-2s_1.\]

Suppose that the polynomial $f(x) = x^6+2x^5+2x^4-2x^2-2x$ takes the square integer value $a^2$ at some $t \in \N$. One calculates that $f = g^2-h$, where
\[ g(x) = x^3+x^2+\frac{x}{2}-\frac{1}{2} \mbox{ and } h(x) = \frac{5}{4}x^2+\frac{3}{2}x+\frac{1}{4}.\]
So
\[(2t^3+2t^2+t-1-2a)(2t^3+2t^2+t-1+2a) = 5t^2+6t+1.\]
This equation in integers is impossible if $t \geq 2$, and so (e) is impossible.

\textit{(f) is impossible:\/} Here, $\hs_1 = s_1^2+s_1+1$ and $\ha = \hs_1^2+1$ and we require that $s_3/s_1$ be a square. As in the above:
\[\frac{s_3-1}{s_1-1} = \hs_2 = (\hs_1+\ha)(\hs-1)+1 = s_1(s_1+1)(\hs_1^2+\hs_1+1) + 1\]
so
\[s_3/s_1 = 1+(s_1^2-1)(\hs_1^2+\hs_1+1) = s_1^6+2s_1^5+3s_1^4+s_1^3-s_1^2-3s_1-2.\]
Suppose the polynomial $f(x) = x^6+2x^5+3x^4+x^3-x^2-3x-2$ takes the square integer value $a^2$ at some $t \in \N$. One calculates that $f = g^2-h$ where
\[g(x) = x^3+x^2+x-\frac{1}{2} \mbox{ and } h(x) = x^2+2x+\frac{9}{4}.\]
So
\[(2t^3+2t^2+2t-1+2a)(2t^3+2t^2+2t-1-2a) = 4t^2+8t+9.\]
This equation is not soluble in integers if $t \geq 2$, and this proves that (f) is not possible.

\textit{(b) is impossible:\/} Consider first the case where $\alpha = s_1(\sqrt{s_1}+1)^2$. Then $\hs_1 = s_1(1+(\sqrt{s_1}+1)^2)$ and $\ha = \hs_1^2-\hs_1$ and we require that $s_3$ must be a square (and so $s_3/s_1$ must be a square). Then:
\[\frac{s_3-1}{s_1-1} = \hs_2 = 1+(\hs_1+\ha)(\hs-1) = 1+(\hs_1-1)\hs_1^2\]
so
\[s_3/s_1 = 1+(s_1-1)(\hs_1-1)s_1(1+(\sqrt{s_1}+1)^2)^2.\]
Put $t = \sqrt{s_1}$. Then $s_3/s_1$ is equal to
\[t^{12}+6t^{11} +17t^{10}+26t^9+17t^8-12t^7-35t^6-28t^5-4t^4+8t^3+4t^2+1.\]
This is equal to $g^2(t) - h(t)$ where
\[g(t) = t^6+3t^5+4t^4+t^3-\frac{5}{2}t^2-\frac{5}{2}t-\frac{1}{2} \mbox{ and }
h(t) = \frac{5}{4}t^4+\frac{7}{2}t^3+\frac{19}{4}t^2+\frac{5}{2}t - \frac{3}{4}.\]
So, as above, we are required to solve the equation
\begin{multline*}
(2t^6+6t^5+8t^4+2t^3-5t^2-5t-1+2a)(2t^6+6t^5+8t^4+2t^3-5t^2-5t-1-2a) \\
= 5t^4+14t^3+19t^2+10t-3
\end{multline*}
with $a, t \in \N$. It is routine (if unpleasant) to check that there are no such solution to this with $t\geq 2$ (for example, $h(x) = 0$ has no solutions with $x \geq 1$, and for $x \geq 2$ we have $g(x) > h(x)$). Thus case (b) cannot occur with $\alpha = s_1(\sqrt{s_1}+1)^2$.

Suppose now that $\alpha = s_1(\sqrt{s_1}-1)^2$. Then the calculations are exactly as for the above case, except that we substitute $-t$ for $t$. Thus we end up having to solve the equation
\begin{multline*}
(2t^6-6t^5+8t^4-2t^3-5t^2+5t-1+2a)(2t^6-6t^5+8t^4-2t^3-5t^2+5t-1-2a) \\
= 5t^4-14t^3+19t^2-10t-3
\end{multline*}
with $a, t \in \N$. Again it is routine to check that $h(-t) \neq 0$ if $t \geq 2$ and that $g(-t) > h(-t)$ if $t \geq 2$. Thus case (b) cannot occur.

This finishes the proof of the Proposition.
\end{proof}

We can now prove the Theorem stated at the beginning.

\medskip

Let $\X = (X,\cl)$ be a locally finite homogeneous geometry of dimension at least 23 with at least 3 points on a line. Let $p,\ell, \Pi$ be (respectively) a point, line and plane in $\X$ with $p \subseteq \ell \subseteq \Pi$. Suppose $\X$ is not a truncation of a projective or affine geometry over a finite field. By Corollary \ref{Cor4} and Lemma 2.1.1 of \cite{DEHG},  one of conditions (1), (2) or (3) of Corollary \ref{Cor4} holds in $\X$, $\X_p$, $\X_{\ell}$ and $\X_\Pi$ (as each has dimension at least 20). We deduce a contradiction using Proposition \ref{Prop5}.

Proposition \ref{Prop5} (parts (d), (e), (f))  implies that condition (3) cannot hold in any of $\X, \X_p$ or $\X_\ell$. If (1) holds in $\X$, then (a) and (b) imply that (3) holds in $\X_p$, but this is impossible, by the above. So (2) holds in $\X$ and therefore (1) holds in $\X_p$ (by (c) and the above). But now (a) and (b) imply that (3) must hold in $\X_\ell$ and we already know that this is not possible, so we have reached a contradiction.

This proves the Theorem.


\begin{thebibliography}{9}

\bibitem{DEthesis} David M. Evans, Some Topics in Group Theory, D. Phil. Thesis, University of Oxford, July 1985.

\bibitem{DEHG} David M. Evans, `Homogeneous geometries', Proc. London Math. Soc. (3) 52 (1986), 305--327.

\bibitem{Z} B. I. Zilber, `Finite homogeneous geometries', in: Proceedings of the 6th Easter Conference in Model Theory, ed. B. Dahn et al. Seminarbericht 98, pp 186--208,  Berlin, Humboldt University, 1988.

\end{thebibliography}
\end{document}